\documentclass[12pt,bezier]{article}
\usepackage{times}
\usepackage{booktabs}
\usepackage{pifont}
\usepackage{floatrow}
\usepackage{caption}
\usepackage{mathrsfs}
\usepackage[fleqn]{amsmath}
\usepackage{amsfonts,amsthm,amssymb,mathrsfs,bbding}
\usepackage{txfonts}
\usepackage{graphics,multicol}
\usepackage{graphicx}
\usepackage{color}

\usepackage{caption}
\usepackage{cite}
\usepackage{latexsym,bm}
\usepackage{indentfirst}
\usepackage{color}
\usepackage[colorlinks=true,anchorcolor=blue,filecolor=blue,linkcolor=blue,urlcolor=blue,citecolor=blue]{hyperref}
\usepackage{extarrows}
\usepackage{cite}
\usepackage{latexsym,bm}
\usepackage{mathtools}
\evensidemargin=\oddsidemargin\topmargin=-1.5cm

\newtheorem{theorem}{Theorem}[section]
\newtheorem{prob}{Problem}[section]

\newtheorem{lemma}{Lemma}[section]

\newtheorem{conj}{Conjecture}[section]

\theoremstyle{definition}

\addtocounter{section}{0}

\begin{document}
\title{A spectral Erd\H{o}s-P\'{o}sa Theorem\footnote{Supported by the National Natural Science Foundation of China (Nos.\,12171066, 11971445),
 Anhui Provincial Natural Science Foundation (No.\,2108085MA13) and Henan Provincial Natural Science Foundation (No.\,202300410377).}}

\author{{\bf Mingqing Zhai$^{a}$},
{\bf Ruifang Liu$^{b}$}\thanks{Corresponding author. E-mail addresses: mqzhai@chzu.edu.cn (M. Zhai); rfliu@zzu.edu.cn (R. Liu).}
\\
{\footnotesize $^a$ School of Mathematics and Finance, Chuzhou University, Chuzhou, Anhui 239012, China} \\
{\footnotesize $^b$ School of Mathematics and Statistics, Zhengzhou University, Zhengzhou, Henan 450001, China}}

\date{}
\maketitle
{\flushleft\large\bf Abstract}
A set of cycles is called independent if no two of them have a common vertex.
Let $S_{n, 2k-1}$ be the complete split graph, which is the join of a clique of size $2k-1$ with an independent set of size $n-2k+1$.
In 1962, Erd\H{o}s and P\'{o}sa established the following edge-extremal result: for every graph $G$ of order $n$ which contains no $k$ independent cycles, where $k\geq2$ and $n\geq 24k$, we have $e(G)\leq (2k-1)(n-k),$ with equality if and only if $G\cong S_{n,2k-1}.$
In this paper, we prove a spectral version of Erd\H{o}s-P\'{o}sa Theorem. Let $k\geq1$ and $n\geq \frac{16(2k-1)}{\lambda^{2}}$ with $\lambda=\frac1{120k^2}$. If $G$ is a graph of order $n$ which contains no $k$ independent cycles, then $\rho(G)\leq \rho(S_{n,2k-1}),$ the equality holds if and only if $G\cong S_{n,2k-1}.$
This presents a new example illustration for which edge-extremal problems have spectral analogues. Finally, a related problem is proposed for further research.

\begin{flushleft}
\end{flushleft}

\section{Introduction}

Given a graph $H$, a graph is said to be \emph{$H$-free},
if it does not contain a subgraph isomorphic to $H$.
The \emph{Tur\'{a}n number} of $H$, denoted by $ex(n,H),$
is the maximum number of edges in an $H$-free graph of order $n.$
An $H$-free graph on $n$ vertices with $ex(n,H)$ edges is called
an \emph{extremal graph} for $H.$
A main purpose of studying Tur\'{a}n numbers is that they are very useful for Ramsey theory, where the original
statements can be seen in \cite{D}.
Tur\'{a}n-type problem can be traced back to Mantel's theorem \cite{MW} in 1907,
which states that $ex(n,C_3)=\lfloor\frac{n^2}4\rfloor$.
In 1941, Tur\'{a}n \cite{bb} showed that $ex(n,K_{k+1})=e(T_{n,k})$ for $n\geq k+1$,
and the $k$-partite Tur\'{a}n graph $T_{n,k}$ is the unique extremal graph for $K_{k+1}$.
Since then, determining the Tur\'{a}n number of a given subgraph $H$ has become a classic problem in extremal graph theory.
However, it is usually very difficult to determine the exact value of $ex(n,H).$

A set of subgraphs is called independent if no two of them have a common vertex.
Particularly, we denote by $kH$ the graph consisting of $k$ independent copies of $H$.
In 1959, Erd\H{o}s and Gallai \cite{EG} determined the maximum number of edges of a graph which contains no $k$ independent edges.
Furthermore, in 1962, Erd\H{o}s \cite{E} proved the maximum number of edges of a graph
which contains no $k$ independent triangles for $n>400(k-1)^{2}$.
Subsequently, Moon \cite{M} showed that Erd\H{o}s's result was still valid whenever
$n>\frac{9}{2}k-\frac{1}{2}.$
Meanwhile, from the perspective of generalizing Tur\'{a}n's theorem,
Simonovits \cite{SIM} and Moon \cite{M} determined the Tur\'{a}n number of $kK_{t}$ for $t\geq3$ and sufficiently large $n$.
Very recently, Chen, Lu and Yuan \cite{CLY} proved the Tur\'{a}n number of $2K_{t}$ for $t\geq3$ and $n\geq 2t.$
The Tur\'{a}n number $ex(n,H)$ is also determined for some other independent subgraphs, see for example, $ex(n,kP_3)$ \cite{BK,LT} and $ex(n,kK_{1,k})$ \cite{LAN}.
More extremal results on independent subgraphs can be seen in \cite{BHC,GER,GOR}.

On the other hand, from the perspective of cycle subgraphs,
F\"{u}redi and Gunderson \cite{FG} determined $ex(n, C_{2k+1})$ for all $k$ and $n$ and characterized the corresponding extremal graphs.
More precisely, $ex(n, C_{2k+1})=\lfloor\frac{n^2}{4}\rfloor$ for $n\geq4k-2$.
However, the exact value of $ex(n, C_{2k})$ is still
open even for $k=2.$ One can refer to \cite{FS} for more results on even cycles.
A variation of Tur\'{a}n-type problem asks what is the maximum number of edges in a graph which contains no $k$ independent cycles.
In 1962, Erd\H{o}s and P\'{o}sa \cite{EP} investigated this extremal problem and characterized the unique extremal graph as follows.

For $n>2k-1$, let $S_{n,2k-1}$ be the join of a clique on $2k-1$ vertices with an independent set of $n-2k+1$ vertices.
This kind of graphs is usually called \emph{complete split graphs},
or \emph{threshold graphs}.

\begin{theorem}[\!\cite{EP}]\label{thm1.1}
Let $k\geq2$ and $n\geq 24k$.
Suppose that $G$ is a graph of order $n$ which contains no $k$ independent cycles. Then
\begin{eqnarray*}\label{eq1}
e(G)\leq (2k-1)(n-k),
\end{eqnarray*}
the equality holds if and only if $G\cong S_{n,2k-1}.$
\end{theorem}

Let $A(G)$ be the \emph{adjacency matrix} of a graph $G$ and $\rho(G)$ be the \emph{spectral radius} of $G$.
By Perron-Frobenius theorem, every connected graph $G$
exists a positive unit eigenvector corresponding to $\rho(G)$,
which is called the \emph{Perron vector} of $G$.
In 1986, Wilf \cite{WILF} showed that $\rho(G)\leq n(1-\frac1k)$ for every $n$-vertex $K_{k+1}$-free graph $G$.
Wilf's result was sharpened by Nikiforov \cite{V1},
who proved a spectral version of Tur\'{a}n's theorem, that is,
$\rho(G)\leq \rho(T_{n,k})$ for every $n$-vertex $K_{k+1}$-free graph $G$,
with equality if and only if $G\cong T_{n,k}$.
Babai and Guiduli \cite{BG} established
asymptotically a K\H{o}v\'{a}ri-S\'{o}s-Tur\'{a}n bound
$\rho(G)\leq \big((s-1)^{\frac1t}+o(1)\big)n^{1-\frac1t}$ for every $n$-vertex $K_{s,t}$-free graph $G$ (where $s\geq t\geq2$),
then Nikiforov \cite{V6} obtained an upper bound with main term $(s-t+1)^{\frac1t}n^{1-\frac1t}$.
In view of the inequality $\rho(G)\geq \frac{2e(G)}n$,
Nikiforov's result is a slight improvement of a result of F\"{u}redi \cite{FU}.
In 2010, Nikiforov \cite{V5} formally posed a spectral version of Tur\'{a}n-type problem.

\begin{prob}\label{prob1}
For a given graph $H$,
what is the maximum spectral radius of an $n$-vertex $H$-free graph?
\end{prob}

In the past decades, much attention has been paid to Problem \ref{prob1} for some specific graphs $H$, see for example, $C_4$ \cite{V1,ZW}, $C_6$ \cite{ZL}, odd cycles \cite{V2},
long cycles \cite{Hou}, edge-disjoint cycles \cite{LZZ}, paths \cite{V5}, 
linear forests \cite{FYZ,CLZ}, star forests \cite{CLZ2}, friendship graphs \cite{CFTZ},
odd wheels \cite{CS}, intersecting odd cycles \cite{LP}, intersecting cliques \cite{DKL},
and minors \cite{HONG,V7,LIN,TT,TT2,ZL1}.

To obtain these results, spectral techniques have been greatly developed.
For example, counting walks \cite{V3}, deleting vertices \cite{V5},
triangle removal lemma \cite{CFTZ}, majorization \cite{LNW}, and so on.
Using the Perron eigenvector (leading eigenvector) to deduce structural properties of spectral extremal graphs (we call this ``leading-eigenvector method"),
Tait and Tobin \cite{TT} successfully solved a spectral extremal conjecture on planar graphs, which was proposed by Boots and Royle \cite{BT},
and independently by Cao and Vince \cite{CAO}.
Very recently, leading-eigenvector method was excellently developed by
Cioab\u{a}, Desai and Tait \cite{CDT,CDT2} to solve the
spectral even cycle conjecture and the spectral Erd\H{o}s-S\'{o}s conjecture,
which were proposed by Nikiforov \cite{V5} in 2010.

Since Problem \ref{prob1} was proposed, the following problem almost
simultaneously abstracted wide interest of scholars.

\begin{prob}\label{prob2}
When does an edge-extremal problem have its spectral analogue, or particularly,
when do they have the same extremal graphs?
\end{prob}

For many Tur\'{a}n-type results, researchers managed to prove their spectral analogues, for example, spectral Tur\'{a}n's theorem \cite{V1}, spectral Erd\H{o}s-Stone-Bollob\'{a}s theorem \cite{ESB}, and
spectral chromatic critical edge theorem (see Theorem 2,\cite{ESB1}).
Moreover, Cioab\u{a}, Desai and Tait \cite{CS} proposed the following conjecture
which extends Nikiforov's spectral chromatic critical edge theorem.

\begin{conj}\label{conj1}
Let $H$ be any graph such that its edge-extremal graphs are Tur\'{a}n graphs plus $O(1)$ edges. Then its spectral extremal family is contained in its edge-extremal family for sufficiently large $n$.
\end{conj}

Conjecture \ref{conj1} has just been solved by Wang, Kang and Xue (see \cite{WKX}).
Inspired by Problems \ref{prob1}-\ref{prob2} and leading-eigenvector method, we investigate a spectral extremal problem on graphs without $k$ independent cycles and establish the following theorem.
Combining Theorem \ref{thm1.1}, we obtain a new example illustration of Problem \ref{prob2}.

\begin{theorem}\label{thm1.2}
Let $k\geq1$ and $n\geq \frac{16(2k-1)}{\lambda^{2}}$ with $\lambda=\frac1{120k^2}$. Assume that $G$ is a graph of order $n$ which contains no $k$ independent cycles.
Then
\begin{eqnarray*}\label{eq111}
\rho(G)\leq \rho(S_{n,2k-1}),
\end{eqnarray*}
the equality holds if and only if $G\cong S_{n,2k-1}.$
\end{theorem}

The rest of the paper is organized as follows. In Section \ref{se2}, we present the proof of Theorem \ref{thm1.2}.
In Section \ref{se3}, a related problem is mentioned for further research.

\section{Proof of Theorem \ref{thm1.2}}\label{se2}

For $k=1,$ every graph is a forest if it does not contain $k$ independent cycles.
We know that the star $S_{n,1}$ attains the maximum spectral radius over all forests of
order $n$. Hence, the theorem holds trivially for $k=1$.
Next we always assume that $k\geq2.$

Let $G^{*}$ be an extremal graph
with maximum spectral radius over all graphs of order $n$ which contain no $k$ independent cycles, and $X$ be the Perron vector of $G^{\ast}.$
In the following, we denote by $\rho$ the spectral radius of $G^{\ast}.$
Moreover, let $u^{\ast}$ be a vertex of $G^{\ast}$ with maximum Perron component, that is, $x_{u^{\ast}}=\max_{u\in V(G)}x_{u}.$
For any vertex $v\in V(G^{*}),$ we denote by $N(v)$ and $N^{2}(v)$ the neighborhood of $v$ and the set of vertices at distance two from $v$, respectively.

\begin{lemma}\label{lem2.1}
If $k\geq2$ and $n\geq2k+3$, then $\rho\geq\sqrt{(2k-1)n}.$
\end{lemma}

\begin{proof}
It is obvious that the complete split graph
$S_{n,2k-1}$ contains no $k$ independent cycles.
Hence, $\rho=\rho(G^{\ast})\geq\rho(S_{n,2k-1}).$
Furthermore,
\begin{eqnarray*}
\rho(S_{n,2k-1})=(k-1)+\sqrt{(k-1)^{2}+(2k-1)(n-2k+1)}.
\end{eqnarray*}
We can check that $\rho(S_{n,2k-1})\geq\sqrt{(2k-1)n}$ for $k\geq2$ and $n\geq2k+3$.
It follows that $\rho\geq\sqrt{(2k-1)n}.$
\end{proof}

In this section, our key goal is to determine the exact size of a vertex subset,
which consists of vertices with large Perron components in $G^{\ast}.$
To do this, we need to introduce two larger vertex subsets $R$ and $R'$.
Let $\lambda=\frac1{120k^2}$
and $R=\{v\in V(G^{\ast}): x_{v}>\lambda x_{u^{\ast}}\}$.
We will present an evaluation on the size of $R$.
For convenience, we define $|R|$ and $e(R)$ as the numbers of vertices
and edges within $R$, respectively. For convenience, we next define $f(n,k)=(2k-1)(n-k)$ in Theorem \ref{thm1.1}.

Before presenting our results, we first introduce the following three formulas which will be used in our arguments.
For any two vertex subsets $A$ and $B$ of $V(G^*)$, we denote by $e(A,B)$
the number of edges from $A$ to $B$. Then we have
\begin{eqnarray}\label{eq-2}
e(A,B)=e(A, B\setminus A)+e(A, A\cap B)
=e(A, B\setminus A)+2e(A\cap B)+e(A\setminus B, A\cap B).
\end{eqnarray}
By (\ref{eq-2}), we can obtain that
\begin{eqnarray}\label{eq-1}
e(A,B)\leq e(A\cup B)+e(A\cap B).
\end{eqnarray}
Moreover, it follows from (\ref{eq-2}) that
\begin{eqnarray}\label{eq0}
e(A,B)\leq |A|\cdot |B|.
\end{eqnarray}

\begin{lemma}\label{lem2.2}
If $k\geq2$ and $n\geq\frac{16(2k-1)}{\lambda^2}$,
then $|R|\leq2\sqrt{(2k-1)n}$.
\end{lemma}

\begin{proof}
For each $v\in R$, it follows from Lemma \ref{lem2.1} and the definition of $R$ that
\begin{eqnarray}\label{eq1}
\lambda\sqrt{(2k-1)n}x_{u^{\ast}}\leq\rho x_{v}=\sum\limits_{u\in N(v)}x_{u}
\leq\sum\limits_{u\in N(v)\cap R}x_{u^{\ast}}+\sum\limits_{u\in N(v)\setminus R}\lambda x_{u^{\ast}}.
\end{eqnarray}
Summing the inequality (\ref{eq1}) for all $v\in R$, we have
\begin{eqnarray}\label{eq2}
|R|\lambda\sqrt{(2k-1)n}x_{u^{\ast}}&\leq&\sum\limits_{v\in R}d_R(v)x_{u^{\ast}}+\lambda\sum\limits_{v\in R}d_{V(G^\ast)\setminus R}(v)x_{u^{\ast}}\nonumber\\
&=&\big(2e(R)+\lambda e(R,V(G^\ast)\setminus R)\big)x_{u^{\ast}}.
\end{eqnarray}
If $|R|\leq2\sqrt{(2k-1)n}$, then we are done.
Otherwise, $|R|>2\sqrt{(2k-1)n}$, then $|R|>24k$
as $n\geq\frac{16(2k-1)}{\lambda^2}$ and $\lambda=\frac1{120k^2}$.
Note that $G^{*}[R]$ contains no $k$ independent cycles.
Thus by Theorem \ref{thm1.1}, $e(R)\leq f(|R|, k)\leq(2k-1)|R|$.
Moreover, we have
$e(R,V(G^\ast)\setminus R)\leq e(G^\ast)\leq(2k-1)n$.
Hence, by (\ref{eq2}) we obtain that
\begin{eqnarray}\label{eq3}
|R|\sqrt{(2k-1)n}\leq \frac2{\lambda}e(R)+e(G^\ast)\leq
\frac2{\lambda}\Big(2k-1\Big)|R|+(2k-1)n.
\end{eqnarray}
Since $n\geq\frac{16(2k-1)}{\lambda^2}$,
we have $\frac2{\lambda}(2k-1)\leq\frac12\sqrt{(2k-1)n}.$
Combining (\ref{eq3}), we obtain that $\frac12|R|\sqrt{(2k-1)n}\leq (2k-1)n$,
which contradicts the assumption $|R|>2\sqrt{(2k-1)n}$.
\end{proof}

Lemma \ref{lem2.2} tells us $|R|=O\big(\!\sqrt{kn}\big)$.
We now need to find a vertex subset $R'$,
which consists of vertices with larger Perron components.
Moreover, we hope that $|R'|=O(k^3)$.
To this end, we define $R'=\{v\in V(G^{\ast}): x_{v}>4\lambda x_{u^{\ast}}\}$.
Clearly, $R'\subseteq R.$
We first give the following lemma,
and then give an evaluation on $|R'|$.

\begin{lemma}\label{lem2.3}
If $k\geq2$, $n\geq\frac{16(2k-1)}{\lambda^{2}}$ and
$v\in V(G^{\ast})$ with $d(v)\leq\frac13\lambda n$, then
\begin{eqnarray*}
e\big(N(v),N^{2}(v)\cap R\big)\leq\Big(\frac53k-1\Big)\lambda n.
\end{eqnarray*}
\end{lemma}

\begin{proof}
Note that $N(v)\cup\big(N^{2}(v)\cap R\big)\subseteq N(v)\cup R$.
If $|N(v)\cup R|<24k$, then
$e\big(N(v),N^{2}(v)\cap R\big)<(12k)^2,$
since $N(v)$ and $N^{2}(v)\cap R$ are disjoint and the sum of their numbers of vertices is less than $24k.$
Hence $e\big(N(v),N^{2}(v)\cap R\big)<\big(\frac53k-1\big)\lambda n$.
Now we assume that $|N(v)\cup R|\geq24k$.
Note that $G^{\ast}[N(v)\cup\{v\}\cup R]$ is a subgraph of $G^{\ast}.$
By Theorem \ref{thm1.1}, we have
$$e\big(N(v)\cup\{v\}\cup R\big)\leq f\big(|N(v)\cup R|+1,k\big)
\leq(2k-1)|N(v)\cup R|.$$
Consequently,
\begin{eqnarray}\label{eq4}
e\big(N(v)\cup R\big) &\leq& (2k-1)|N(v)\cup R|-d(v)
\leq(2k-2)d(v)+(2k-1)|R|\nonumber\\
&\leq& \frac13\big(2k-2\big)\lambda n+(2k-1)|R|.
\end{eqnarray}
By Lemma \ref{lem2.2}, we have $|R|\leq2\sqrt{(2k-1)n}$. Hence,
$$(2k-1)|R|\leq2(2k-1)\sqrt{(2k-1)n}\leq\frac12\big(2k-1\big)\lambda n,$$
where the last inequality follows from $n\geq\frac{16(2k-1)}{\lambda^{2}}$.
Now by (\ref{eq4}), we have $e\big(N(v),N^{2}(v)\cap R\big)\leq e\big(N(v)\cup R\big)\leq\big(\frac53k-1\big)\lambda n,$ as desired.
\end{proof}

\begin{lemma}\label{lem2.4}
If $k\geq2$ and $n\geq\frac{16(2k-1)}{\lambda^{2}}$,
then $|R'|\leq\frac{6k}\lambda,$
and $d(v)>\frac13\lambda n$ for each $v\in R'$.
\end{lemma}

\begin{proof}
We first show that $d(v)>\frac13\lambda n$ for each $v\in R'$.
Suppose to the contrary that
there exists a vertex $v\in R'$ such that $d(v)\leq\frac13\lambda n$.
By Lemma \ref{lem2.1} and the definition of $R'$, we have
\begin{eqnarray}\label{eq5}
\rho^{2}x_{v}\geq(2k-1)nx_{v}>4(2k-1)\lambda nx_{u^{\ast}}.
\end{eqnarray}
Note also that
\begin{eqnarray}\label{eq6}
\rho^{2}x_{v}&=&d(v)x_{v}+\sum\limits_{u\in N(v)}d_{N(v)}(u)x_{u}+\sum\limits_{u\in N^{2}(v)}d_{N(v)}(u)x_{u}\nonumber\\
&\leq&d(v)x_{u^{\ast}}+2e\big(N(v)\big)x_{u^{\ast}}+\sum\limits_{u\in N^{2}(v)\cap R}d_{N(v)}(u)x_{u}
+\sum\limits_{u\in N^{2}(v)\setminus R}d_{N(v)}(u)x_{u}.
\end{eqnarray}
Note that $x_v>4\lambda x_{u^{\ast}}$ and $\rho\geq\sqrt{(2k-1)n}$.
Then $4\lambda\sqrt{(2k-1)n} x_{u^{\ast}}<\rho x_v=\sum_{u\in N(v)}x_u\leq d(v)x_{u^{\ast}}.$ This implies that $d(v)>4\lambda\sqrt{(2k-1)n}\geq24k$.
Since $G^{\ast}[N(v)\cup\{v\}]$ contains no $k$ independent cycles,
by Theorem \ref{thm1.1} we have
$$e\big(N(v)\cup\{v\}\big)\leq f\big(d(v)+1,k\big)
=(2k-1)\big(d(v)+1-k\big)\leq (2k-1)d(v),$$
and so $e\big(N(v)\big)\leq (2k-2)d(v)$.
It follows that
\begin{eqnarray}\label{eq7}
\Big(d(v)+2e\big(N(v)\big)\Big)x_{u^{\ast}}
\leq(4k-3)d(v)x_{u^{\ast}}\leq\Big(\frac43k-1\Big)\lambda nx_{u^{\ast}},
\end{eqnarray}
as $d(v)\leq\frac13\lambda n$.
According to Lemma \ref{lem2.3}, we obtain that
\begin{eqnarray}\label{eq8}
\sum\limits_{u\in N^{2}(v)\cap R}d_{N(v)}(u)x_{u}\leq e\big(N(v),N^{2}(v)\cap R\big)x_{u^{\ast}}\leq\Big(\frac53k-1\Big)\lambda nx_{u^{\ast}}.
\end{eqnarray}
Furthermore, by Theorem \ref{thm1.1} we have $e\big(N(v),N^{2}(v)\setminus
R\big)\leq e(G^{\ast})\leq f(n,k).$
According to the definition of $R$, it follows that
\begin{eqnarray}\label{eq9}
\sum\limits_{u\in N^{2}(v)\setminus R}d_{N(v)}(u)x_{u}\leq e\big(N(v),N^{2}(v)\setminus
R\big)\lambda x_{u^{\ast}}\leq f(n,k)\lambda x_{u^{\ast}}\leq(2k-1)\lambda nx_{u^{\ast}}.
\end{eqnarray}
Combining the inequalities (\ref{eq6})-(\ref{eq9}), we have
\begin{eqnarray}\label{eq10}
\rho^{2}x_{v}\leq \Big(\big(\frac43k-1\big)\lambda n+
\big(\frac53k-1\big)\lambda n+(2k-1)\lambda n\Big)x_{u^{\ast}}
=(5k-3)\lambda n x_{u^{\ast}}.
\end{eqnarray}
Comparing (\ref{eq10}) with (\ref{eq5}),
we obtain that $4(2k-1)<5k-3$,
a contradiction.
Therefore, $d(v)>\frac13\lambda n$ for each $v\in R'$.

Next, we prove $|R'|\leq\frac{6k}\lambda.$
If $|R'|\leq24k$, then we are done as $\lambda=\frac1{120k^2}$.
Assume that $|R'|\geq24k$.
By Theorem \ref{thm1.1},
$e(R')\leq f(|R'|,k)\leq (2k-1)|R'|$. Moreover,
\begin{eqnarray*}
(2k-1)n\geq f(n,k)\geq e(G^{\ast})
=\frac12\sum\limits_{v\in R'}d(v)+\frac12\sum\limits_{v\in V(G^{\ast})\setminus R'}d(v),
\end{eqnarray*}
where $\sum_{v\in R'}d(v)\geq \frac13\lambda n|R'|$ and
$$\sum\limits_{v\in V(G^{\ast})\setminus R'}d(v)
\geq e\big(R',V(G^{\ast})\setminus R'\big)=\sum\limits_{v\in R'}d(v)-2e(R')
\geq \frac13\lambda n|R'|-2(2k-1)|R'|.$$
Combining the above three inequalities, we obtain that
$(2k-1)n\geq \frac13\lambda n|R'|-(2k-1)|R'|.$
It follows that
$$|R'|\leq \frac{2kn-n}{\frac13\lambda n-2k+1}
\leq\frac{2kn}{\frac13\lambda n}=\frac{6k}\lambda,$$
where the second inequality follows from $\lambda=\frac{1}{120k^2}$ and $n\geq\frac{16(2k-1)}{\lambda^{2}}$.
\end{proof}

Recall that $R=\{v\in V(G^{\ast}): x_{v}>\lambda x_{u^{\ast}}\}$
and $R'=\{v\in V(G^{\ast}): x_{v}>4\lambda x_{u^{\ast}}\}$,
where $\lambda=\frac1{120k^2}$.
Now we consider a vertex subset
which consists of vertices with larger Perron components.
Let $R''=\{v\in V(G^{\ast}): x_{v}\geq\frac1{4k}x_{u^{\ast}}\}$.
Clearly, $R''\subseteq R'\subseteq R$.
We shall prove that $|R''|=O(k)$.
Before this, we need the following three lemmas.

\begin{lemma}\label{lem2.5}
Let $k\geq2$ and $n\geq\frac{16(2k-1)}{\lambda^{2}}$.
Then for each vertex $v\in R'$, we have
\begin{eqnarray*}
e\big(N(v)\setminus R',R'\setminus\{v\}\big)\leq(2k-2)d(v)+(2k-1)|R'|.
\end{eqnarray*}
\end{lemma}

\begin{proof}
Consider two arbitrary vertices $v_{1},v_{2}\in N(v)\setminus R'.$
If they always have no common neighbours in $R'\setminus\{v\}$, then $$e\big(N(v)\setminus R', R'\setminus\{v\}\big)\leq|R'\setminus\{v\}|\leq|R'|\leq(2k-2)d(v)+(2k-1)|R'|$$
for $k\geq1,$ as desired.
Otherwise there exist two vertices $v_{1}, v_{2}\in N(v)\setminus R'$, which have a common neighbour $v_{0}$ in $R'\setminus\{v\}.$
Then $vv_{1}v_{0}v_{2}$ forms a quadrilateral.
Let $G=G^{\ast}[N(v)\setminus\big(R'\cup\{v_{1},v_{2}\}\big), R'\setminus\{v,v_{0}\}]$ be the bipartite subgraph of $G^{\ast}$, whose edge set consists of the edges from $N(v)\setminus\big(R'\cup\{v_{1},v_{2}\}\big)$ to $R'\setminus\{v,v_{0}\}.$ Then $G$ contains at most $k-2$ independent cycles.
By Lemma \ref{lem2.4},
we have $d(v)>\frac13\lambda n>48k^2$,
and so $|G|=|N(v)\cup R'|-4>24k$.
Thus by Theorem \ref{thm1.1}, we have
\begin{eqnarray*}
e(G)\leq f(|G|,k-1)=(2k-3)(|G|-k+1)\leq (2k-3)|G|\leq (2k-3)\big(|N(v)|+|R'|\big).
\end{eqnarray*}
It follows that
\begin{eqnarray*}
e\big(N(v)\setminus R', R'\setminus\{v\}\big)&=&e(G)+\sum_{i=1}^{2}d_{R'\setminus\{v\}}(v_{i})+d_{N(v)\setminus R'}(v_{0})-2\\
&\leq& e(G)+2|R'|+|N(v)\setminus R'|\\
&\leq&(2k-2)d(v)+(2k-1)|R'|.
\end{eqnarray*}
This completes the proof.
\end{proof}

Recall that $x_v\geq\frac1{4k}x_{u^*}$ for
each vertex $v\in R''$.
The following lemma indicates that the vertices in $R''$ have high degrees.

\begin{lemma}\label{lem2.6}
Let $k\geq2$ and $n\geq\frac{16(2k-1)}{\lambda^{2}}$.
If $v\in R''$ with $x_{v}=\mu x_{u^{\ast}}$,
then $d(v)>\big(\mu-\frac{1}{12k}\big)n$.
\end{lemma}

\begin{proof}
By the definition of $R''$ and $x_{v}=\mu x_{u^{\ast}}$, we know that $\frac{1}{4k}\leq\mu\leq1$.
Suppose to the contrary that there exists a vertex $v\in R''$ with $d(v)\leq\big(\mu-\frac{1}{12k}\big)n.$
Then by Lemma \ref{lem2.1}, we have
\begin{eqnarray}\label{eq11}
\rho^{2}x_{v}\geq(2k-1)nx_{v}.
\end{eqnarray}
Denote $M=N(v)\cup N^2(v)$.
Since $x_u\leq4\lambda x_{u^*}$ for each $u\in M\setminus R'$, we have
\begin{eqnarray}\label{eq12}
\rho^{2}x_{v} &=& d(v)x_{v}+\sum\limits_{u\in M\setminus R'}d_{N(v)}(u)x_{u}+\sum\limits_{u\in M\cap R'}d_{N(v)}(u)x_{u}\nonumber\\
&\leq & d(v)x_{v}+e\big(M\setminus R',N(v)\big)\cdot4\lambda x_{u^{\ast}}+e\big(M\cap R',N(v)\big)x_{u^{\ast}}.
\end{eqnarray}
Note that $(M\setminus R')\cap N(v)=N(v)\setminus R'.$
By (\ref{eq-1}), we obtain that
\begin{eqnarray}\label{eq13}
e\big(M\setminus R',N(v)\big)\leq e(M)+e\big(N(v)\setminus R'\big)
\leq 2e(G^*)\leq 2f(n,k)\leq 2(2k-1)n.
\end{eqnarray}
On the other hand, since $M\cap R'\subseteq R'\setminus\{v\}$, it follows from ({\ref{eq-2}}) that
\begin{eqnarray}\label{eq14}
e\big(M\cap R',N(v)\big)&\leq& e\big(R'\setminus\{v\},N(v)\big)\nonumber\\
&\leq& e\big(R'\setminus\{v\},N(v)\setminus R'\big)+2e\big(R'\cap N(v)\big)+e\big(R'\setminus N(v), R'\cap N(v)\big)\nonumber\\
&\leq& (2k-2)d(v)+(2k-1)|R'|+2e(R'),
\end{eqnarray}
where the last inequality follows from Lemma \ref{lem2.5} and $e\big(R'\cap N(v)\big)+e\big(R'\setminus N(v), R'\cap N(v)\big)\leq e(R').$
If $|R'|\geq 24k$, then
by Theorem \ref{thm1.1}, $e(R')\leq (2k-1)|R'|,$ and by
Lemma \ref{lem2.4}, we have
\begin{eqnarray}\label{eq15}
(2k-1)|R'|+2e(R')\leq 3(2k-1)\frac{6k}{\lambda}.
\end{eqnarray}
Note that $\lambda\leq\frac1{120k}$.
If $|R'|<24k$, then $e(R')<{24k\choose 2}$ and
(\ref{eq15}) also holds.
Now combining the inequalities (\ref{eq11})-(\ref{eq15}), we get that
\begin{eqnarray*}
\Big((2k-1)n-d(v)\Big)x_{v}\leq\Big(\rho^{2}-d(v)\Big)x_{v}
\leq\Big(2(2k-1)n\cdot4\lambda+18k\big(2k-1\big)\frac{1}{\lambda}+(2k-2)d(v)\Big)x_{u^{\ast}}.
\end{eqnarray*}
Recall that $x_{v}=\mu x_{u^{\ast}}$ and $d(v)\leq\big(\mu-\frac{1}{12k}\big)n$.
Then
\begin{eqnarray*}
\Big((2k-1)n-\big(\mu-\frac{1}{12k}\big)n\Big)\mu\leq8(2k-1)\lambda n
+18k\big(2k-1\big)\frac{1}{\lambda}+(2k-2)\big(\mu-\frac{1}{12k}\big)n.
\end{eqnarray*}
Equivalently, we have
\begin{eqnarray}\label{eq18}
\Big(\mu-\mu^{2}+\frac{\mu+2k-2}{12k}\Big)n\leq8(2k-1)\lambda n+18k\big(2k-1\big)\frac{1}{\lambda}.
\end{eqnarray}
Let $g(\mu)=\mu-\mu^{2}+\frac{\mu+2k-2}{12k}$,
where $\frac{1}{4k}\leq\mu\leq1$.
Clearly, $g(\mu)$ is a convex function on $\mu$. Hence
\begin{eqnarray}\label{eq19}
\min\limits_{\frac{1}{4k}\leq\mu\leq1}g(\mu)
=\min\Big\{g\big(\frac{1}{4k}\big),g(1)\Big\}=\frac{2k-1}{12k}.
\end{eqnarray}
Combining (\ref{eq18}) and (\ref{eq19}), we can see that
\begin{eqnarray*}
\frac{(2k-1)n}{12k}\leq8(2k-1)\lambda n+18k\big(2k-1\big)\frac{1}{\lambda},
\end{eqnarray*}
that is, $\frac{n}{12k}\leq8\lambda n+\frac{18k}{\lambda}.$
One can easily check that this
contradicts the assumption $n\geq\frac{16(2k-1)}{\lambda^{2}}$,
since $\lambda=\frac{1}{120k^2}$.
\end{proof}

By Lemma \ref{lem2.6} and the definition of $R''$,
we know that $d(v)>\big(\frac{1}{4k}-\frac{1}{12k}\big)n=\frac{1}{6k}n$
for each $v\in R''$.
We now present a more refined evaluation on the degrees of vertices in $R''$.
We will observe that there exists no a vertex $v\in R''$ with
$d(v)\in \big(\frac{1}{6k}n,(1-\frac{5}{12k})n\big)$.

\begin{lemma}\label{lem2.7}
Let $k\geq2$ and $n\geq\frac{16(2k-1)}{\lambda^{2}}$. Then
$d(v)\geq(1-\frac{5}{12k})n$ for each $v\in R''$.
\end{lemma}

\begin{proof}
By Lemma \ref{lem2.6}, it suffices to show that $x_{v}\geq\big(1-\frac{1}{3k}\big)x_{u^{\ast}}$
for each $v\in R''$.
Suppose to the contrary that there exists a vertex $v\in R''$
with $x_{v}=\nu x_{u^*}<\big(1-\frac{1}{3k}\big)x_{u^{\ast}}$.

Denote $M=V(G^*)\setminus\{u^*\}$.
Recall that $x_u>4\lambda x_{u^*}$ for each $u\in R'$. Then
\begin{eqnarray}
\rho^{2}x_{u^{\ast}} &=& d(u^{\ast})x_{u^{\ast}}+\sum\limits_{u\in M\cap R'}d_{N(u^{\ast})}(u)x_{u}+\sum\limits_{u\in M\setminus R'}d_{N(u^{\ast})}(u)x_{u}\label{eq20}\\
&\leq & d(u^{\ast})x_{u^{\ast}}+e\big(M\cap R',N(u^*)\big)x_{u^{\ast}}+
e\big(M\setminus R',N(u^*)\big)\cdot4\lambda x_{u^{\ast}}.\label{eq21}
\end{eqnarray}
Since $v\in R''\setminus\{u^*\}$ and $R''\subseteq R'$,
we have $v\in M\cap R'$. One can see that
$x_v$ appears $d_{N(u^*)}(v)$ times
in (\ref{eq20}), and is enlarged to $x_{u^*}$ in (\ref{eq21}).
Thus,
\begin{eqnarray}\label{eq22}
\rho^{2}x_{u^{\ast}}\leq \Big(d(u^{\ast})+e\big(M\cap R',N(u^*)\big)\Big)x_{u^{\ast}}
+e\big(M\setminus R',N(u^*)\big)\cdot4\lambda x_{u^{\ast}}+d_{N(u^{\ast})}(v)(x_{v}-x_{u^{\ast}}).
\end{eqnarray}
Note that $R'\setminus M=\{u^*\}$. Then we have $d(u^{\ast})+e\big(M\cap R',N(u^*)\big)=e\big(R',N(u^{\ast})\big).$
Moreover, notice that $R'\cup (M\setminus R')=V(G^{*}),$ then by (\ref{eq-1}) and Theorem \ref{thm1.1}, we have
\begin{eqnarray}\label{eq24}
e\big(R',N(u^{\ast})\big)+e\big(M\setminus R',N(u^*)\big)
=e\big(V(G^{*}), N(u^*)\big)\leq e\big(G^{*})+e\big(N(u^*)\big)\leq2f(n,k).
\end{eqnarray}
It follows from (\ref{eq24}) that
\begin{eqnarray}\label{eq-24}
&&d(u^{\ast})+e\big(M\cap R',N(u^*)\big)+e\big(M\setminus R',N(u^*)\big)\cdot4\lambda \nonumber\\
&\leq& e\big(R',N(u^{\ast})\big)+\Big(2f(n,k)-e\big(R',N(u^{\ast})\big)\Big)\cdot4\lambda \nonumber\\
&=&(1-4\lambda)e\big(R',N(u^{\ast})\big)+8\lambda f(n,k).
\end{eqnarray}
By (\ref{eq-1}) and Theorem \ref{thm1.1}, we have
\begin{eqnarray}\label{eq23}
e\big(R',N(u^{\ast})\big)\leq e\big(R'\cup N(u^*)\big)+e\big(R'\cap N(u^*)\big)
\leq e(G^{*})+e(R')\leq f(n,k)+e(R').
\end{eqnarray}
Combining (\ref{eq22})-(\ref{eq23}), we can see that
$$\rho^{2}x_{u^{\ast}}
\leq \Big((1+4\lambda)f(n,k)+e(R')\Big) x_{u^{\ast}}+d_{N(u^{\ast})}(v)(x_{v}-x_{u^{\ast}}).$$
Recall that
$f(n,k)\leq(2k-1)n$, and $x_v=\nu x_{u^{\ast}}$, where $\frac1{4k}\leq v<1-\frac1{3k}$.
Moreover, by Lemma \ref{lem2.4}, $|R'|\leq\frac{6k}\lambda$,
and so $e(R')\leq \max\Big\{(2k-1)|R'|,{24k\choose 2}\Big\}
\leq(2k-1)\frac{6k}\lambda.$ Hence
$$\rho^{2}x_{u^{\ast}}
\leq \Big((1+4\lambda)(2k-1)n+(2k-1)\frac{6k}\lambda\Big) x_{u^{\ast}}+d_{N(u^{\ast})}(v)(\nu-1)x_{u^{\ast}}.$$
Combining $\rho^2\geq(2k-1)n$, we can get that
\begin{eqnarray}\label{eq25}
(1-\nu)d_{N(u^{\ast})}(v)\leq(2k-1)\frac{6k}\lambda+4(2k-1)\lambda n.
\end{eqnarray}
Moreover, by Lemma \ref{lem2.6}, $d(u^*)>\big(1-\frac1{12k}\big)n$
and $d(v)>\big(\nu-\frac1{12k}\big)n.$
It follows that $d_{N(u^{\ast})}(v)>\big(\nu-\frac1{6k}\big)n.$
Combining (\ref{eq25}), we have
\begin{eqnarray}\label{eq26}
(1-\nu)\Big(\nu-\frac1{6k}\Big)n-4(2k-1)\lambda n<\Big(2k-1\Big)\frac{6k}\lambda.
\end{eqnarray}
Define $h(\nu)=(1-\nu)\big(\nu-\frac1{6k}\big)$.
Then $h(\nu)$ is a convex function on $\nu$.
Hence, $(1-\nu)\big(\nu-\frac1{6k}\big)
\geq\min\{h\big(\frac{1}{4k}\big),h\big(1-\frac{1}{3k}\big)\}
=\frac1{12k}-\frac1{48k^2}.$
On the other hand, recall that $\lambda=\frac{1}{120k^2}$,
then $4(2k-1)\lambda\leq 8k\lambda=\frac1{15k}$.
Consequently,
\begin{eqnarray*}
(1-\nu)\Big(\nu-\frac1{6k}\Big)n-4(2k-1)\lambda n
\geq\Big(\frac1{12k}-\frac1{48k^2}-\frac1{15k}\Big)n\geq\frac{n}{160k}.
\end{eqnarray*}
Combining (\ref{eq26}), we have
$\frac{n}{160k}<\Big(2k-1\Big)\frac{6k}\lambda,$
which contradicts the fact that $n\geq \frac{16(2k-1)}{\lambda^2}$.
\end{proof}

We now prove that $|R''|=O(k)$.
In fact, we prove an exact value of $|R''|$.
To do it, we also need to use the following inequality,
which can be proved by induction or double counting.

\begin{lemma}\label{lem2.8}
If $S_{1},\ldots,S_{k}$ are $k$ finite sets, then $|S_{1}\cap\cdots\cap S_{k}|\geq \sum_{i=1}^{k}|S_{i}|-(k-1)|\cup_{i=1}^{k}S_{i}|.$
\end{lemma}

\begin{lemma}\label{lem2.9}
Let $k\geq2$ and $n\geq\frac{16(2k-1)}{\lambda^{2}}$. Then $|R''|=2k-1$.
\end{lemma}

\begin{proof}
We first prove $|R''|\leq2k-1$.
Suppose to the contrary that $|R''|\geq2k.$
We choose $2k$ vertices $v_1,v_2,\ldots,v_{2k}\in R''$.
By Lemmas \ref{lem2.7} and \ref{lem2.8},
\begin{eqnarray*}
|N(v_1)\cap\cdots\cap N(v_{2k})| &\geq& \sum_{i=1}^{2k}d(v_{i})-(2k-1)|N(v_{1})\cup\cdots\cup N(v_{2k})|\\
&>& 2k\Big(1-\frac{5}{12k}\Big)n-(2k-1)n=\frac{n}{6}.
\end{eqnarray*}
Note that $n>12k.$ Then we can find a subgraph $K_{2k,2k}$ in $G^{\ast},$ and hence $G^{\ast}$ contains $k$ independent quadrilaterals, a contradiction. Therefore, $|R''|\leq 2k-1$.

Next we will show that $|R''|\geq2k-1$. Suppose to the contrary that $|R''|\leq2k-2$.
Define $M=N(u^*)\cup N^2(u^*)$. Then by the definition of $R''$, we have
\begin{eqnarray}
\rho^{2}x_{u^{\ast}} &=& d(u^{\ast})x_{u^{\ast}}+\sum\limits_{u\in M\cap R''}d_{N(u^{\ast})}(u)x_{u}+\sum\limits_{u\in M\setminus R''}d_{N(u^{\ast})}(u)x_{u}\label{eq27}\\
&\leq & d(u^{\ast})x_{u^{\ast}}+e\big(M\cap R'',N(u^*)\big)x_{u^{\ast}}+
e\big(M\setminus R'',N(u^*)\big)\frac{1}{4k}x_{u^{\ast}}.\label{eq28}
\end{eqnarray}
Note that $N(u^*)\subseteq M$. By (\ref{eq-1}) and Theorem \ref{thm1.1}, we have
$$e\big(M\setminus R'',N(u^*)\big)\leq e(M)+e\big(N(u^*)\setminus R''\big)
\leq 2e(G^*)\leq 2f(n,k)\leq2(2k-1)n.$$
On the other hand, notice that $M\cap R''\subseteq R''\setminus\{u^*\}$, then by (\ref{eq0}) we have $$e\big(M\cap R'',N(u^*)\big)\leq |M\cap R''|\cdot|N(u^*)|\leq|R''\setminus\{u^*\}|\cdot d(u^*).$$
Combining again (\ref{eq27}) and (\ref{eq28}), we obtain
\begin{eqnarray*}
\rho^{2}x_{u^{\ast}}
&\leq& |R''|\cdot d(u^{\ast})x_{u^{\ast}}
+2(2k-1)n\frac{1}{4k}x_{u^{\ast}}.
\end{eqnarray*}
Recall that $\rho^{2}\geq(2k-1)n$ and $|R''|\leq2k-2$. Hence we get that
\begin{eqnarray*}
(2k-1)n\leq \rho^2\leq(2k-2)n+\frac{2k-1}{2k}n,
\end{eqnarray*}
which gives $n\leq \frac{2k-1}{2k}n$,
a contradiction. Therefore, $|R''|\geq2k-1$. This completes the proof.
\end{proof}

Now, we are ready to give the final proof of Theorem \ref{thm1.2}.

\medskip
\noindent
\textbf{Proof of Theorem \ref{thm1.2}.}
\vspace{2mm}

Let $R'''=\big\{v\in V(G^{\ast}): R''\subseteq N(v)\big\}$. Then $R'''\cap R''=\varnothing$.
Moreover, $G^{\ast}$ contains $K_{2k-1,|R'''|}$ as a subgraph,
since $|R''|=2k-1$ by Lemma \ref{lem2.9}.
Furthermore, by Lemma \ref{lem2.8}, we have $d(u)\geq(1-\frac{5}{12k})n$ for each $u\in R''$. Hence
\begin{eqnarray*}
|R'''|\geq n-\frac{5}{12k}n\times|R''|>n-\frac{5}{6}n=\frac{n}{6}.
\end{eqnarray*}
Now we find a large complete bipartite subgraph $K_{2k-1,\frac n6}$ in $G^*[R'',R''']$.
This implies that we can obtain $k-1$ independent quadrilaterals,
which arbitrarily use $2k-2$ vertices in $R''$ and $2k-2$ vertices in $R'''$.

Let $R''''=V(G^{\ast})\setminus(R''\cup R''')$. In the following, we will characterize the structure of $G^{\ast}[R'''\cup R'''']$.
More precisely, we shall prove that $R'''\cup R''''$ is an independent set.
The proof is divided into three claims.

\vspace{2mm}
\noindent
{\bf Claim 1. $G^{\ast}[R'''\cup R'''']$ contains no a path such that its two endpoints have common neighbors in $R''$.}
\vspace{2mm}

Suppose to the contrary that there exists a path $P=v_1v_2\cdots v_s$ in $G^{\ast}[R'''\cup R'''']$ with $N(v_1)\cap N(v_s)\cap R''\neq\varnothing$.
We may assume that $P$ is shortest. Then $v\notin R'''$ for each $v\in V(P)\setminus\{v_1,v_s\}$.
Now, let $v_{0}$ be an arbitrary vertex in $R''$, then $P+v_{0}v_{1}+v_{0}v_{s}$ forms a cycle. Note that $|R''|=2k-1$ and $|R'''|>\frac n6$.
Whether $v_1,v_s\in R'''$ or not, $G^{\ast}[R''\setminus\{v_{0}\},R'''\setminus V(P)]$ always contains a subgraph isomorphic to $K_{2k-2,2k-2}$.
Hence there exist $k-1$ vertex-disjoint quadrilaterals in $G^{\ast}[R''\setminus\{v_{0}\},R'''\setminus V(P)]$.
Consequently, we find $k$ independent cycles in $G^*$, a contradiction.

\vspace{2mm}
\noindent
{\bf Claim 2. $G^{\ast}[R'''\cup R'''']$ is acyclic.}
\vspace{2mm}

Suppose to the contrary that there exists a cycle $C\subseteq G^{\ast}[R'''\cup R'''']$.
If $|V(C)\cap R'''|\leq1$, then $C$ and $k-1$ vertex-disjoint quadrilaterals in $G^{*}[R'', R''']$ form $k$ independent cycles in $G^{*}$, a contradiction.
If $|V(C)\cap R'''|\geq2$, say, $v_{1},v_{2}\in V(C)\cap R'''$,
then we can find a path $P$ from $v_{1}$ to $v_{2}$ along the cycle $C$.
Since $R''\subseteq N(v_1)\cap N(v_2)$, we obtain a contradiction to Claim 1.

\vspace{2mm}
\noindent
{\bf Claim 3. $G^{\ast}[R'''\cup R'''']$ contains no vertices of degree one.}
\vspace{2mm}

Otherwise, assume that there exists a vertex $v_{3}\in R'''\cup R''''$ with
$d_{R'''\cup R''''}(v)=1$. Assume that $v_{4}$ is the unique neighbour of $v_{3}$ in $R'''\cup R''''$.

If $v_{3}\in R''''$, that is, $v_{3}\notin R'''$, then there exists a vertex $v_{5}\in R''$ such that $v_{3}$ is not adjacent to $v_{5}$. By the definition of $R''$, we have $x_{v_{5}}\geq\frac{x_{u^{\ast}}}{4k}>x_{v_{4}}$.
Let $G=G^{\ast}-v_{3}v_{4}+\{v_{3}v: v\in R''\setminus N(v_{3})\}$.
Then $N_G(v_3)=R''$.
We can conclude that $G$ also has no $k$ independent cycles.
To see this, assume that $G$ contains a family $\mathbb{C}$ of $k$ independent cycles,
then $v_3$ must belong to some $C\in \mathbb{C}$.
Moreover, $C$ uses at least two vertices in $R''$.
Thus $|R''\setminus V(C)|\leq 2k-3$.
However, by Claim 2, $G^*[R'''\cup R''''$] is acyclic,
and so $G[R'''\cup R'''']$ is acyclic.
Hence each cycle in $\mathbb{C}$ uses at least one vertex in $R''$,
which implies that there exists one cycle $C'\in\mathbb{C}$
with $|V(C')\cap R''|=1.$
This contradicts Claim 1.

On the other hand,
\begin{eqnarray*}
\rho(G)-\rho(G^{\ast}) &\geq& X^{T}(A(G)-A(G^{\ast}))X
=2x_{v_{3}}\Big(\sum_{v\in R''\setminus N(v_{3})}x_{v}-x_{v_{4}}\Big)\\
&\geq& 2x_{v_{3}}(x_{v_{5}}-x_{v_{4}})>0,
\end{eqnarray*}
which contradicts the maximality of $\rho(G^{\ast})$.

If $v_{3}\in R'''$,
then there exists a longest path $P'=v_3v_4\cdots v_t$ starting at $v_3$
in $G^*[R'''\cup R''''].$
Since $G^*[R'''\cup R'''']$ is acyclic by Claim 2,
we have $d_{R'''\cup R''''}(v_t)=1$.
By the previous proof, it follows that $v_t\notin R''''$.
Thus $v_t\in R'''.$
But now, $P'$ is a path of $G^*[R'''\cup R'''']$ with both endpoints
in $R'''$, which contradicts Claim 1.

\vspace{2mm}

Combining Claim 2 and Claim 3, we conclude that
$R'''\cup R''''$ is an independent set.
Note that $G^{\ast}$ is edge-maximal
(otherwise, adding edges will increase the spectral radius).
Therefore, $R''$ must be a clique and
$e(R'',R'''\cup R'''')=|R''|\cdot|R'''\cup R''''|.$ This implies that $G^{\ast}\cong S_{n,2k-1}$.
This completes the proof.
\hspace*{\fill}$\Box$

\section{Concluding remarks}\label{se3}

%
%

In this paper, we prove a spectral version of Erd\H{o}s-P\'{o}sa theorem,
more precisely, we determine the unique spectral extremal graph over all $n$-vertex
graphs which contain no $k$ independent cycles.
To close this paper, we propose a general problem for further research.

\begin{prob}\label{prob1}
Given a positive integer $k$ and a connected graph $H$ on $\ell$ vertices.
What is the maximum spectral radius of an $n$-vertex graph which contains no $k$ independent copies of $H$?
\end{prob}

Problem \ref{prob1} was solved for some specific graphs $H$, see for example,
$K_2$ \cite{FYZ}, the path $P_\ell$ \cite{CLZ},
the star $S_\ell$ \cite{CLZ2}.
By a private communication, we know that B. Ning and Z.W. Wang recently solved the case that $H$ is a triangle. These results used different methods.
We would like to see more results on Problem \ref{prob1} and more methods to
deal with specific questions.


\end{document}